\documentclass[10pt]{amsart}
\usepackage{geometry}                
\usepackage{fullpage}

\usepackage{Adrians_standard_stylesheet}
\usepackage{quiver}

\theoremstyle{plain}
\newtheorem{theorem}{Theorem}[section]
\newtheorem{corollary}[theorem]{Corollary}

\newtheorem{proposition}[theorem]{Proposition}  
\newtheorem{example}[theorem]{Example}

\theoremstyle{definition}
\newtheorem{definition}[theorem]{Definition}

\theoremstyle{remark}
\newtheorem{remark}[theorem]{Remark}
\newtheorem{question}[theorem]{Question}

\author
{
Adrian Clough
}
\thanks{
The author acknowledges support by \emph{Tamkeen} under \emph{NYUAD Research Institute grant} \texttt{CG008}.
}
\date{\today}                                           

\title{Local rigidity and six functor formalisms}

\begin{document}

\maketitle

\begin{abstract} 
The coefficient categories of six functor formalisms are often locally rigid, and when this is the case, the exceptional pushforward and pullback adjunctions may be defined formally. In this short note it is shown that for $f$ a proper map resp.\ an open embedding the well known formulas $f_! \simeq f_*$ resp.\ $f_! \simeq f_\sharp$ may likewise be deduced formally. 
\end{abstract} 

Consider a site $C$ of suitable geometric objects such as schemes, manifolds, locally compact Hausdorff spaces, \ldots, then a six functor formalism on $C$ consists, roughly speaking, of a lax symmetric monoidal sheaf on $\Corr_C$ valued in $\mathbf{Pr}^L$ satisfying well-known properties expressing various geometrically meaningful dualities (see e.g., \cite{aK2023} or \cite{pS2023}).  In \cite[Th.~5.5.5.1]{jL2012} Lurie constructs for any locally compact Hausdorff space $X$ an equivalence between the symmetric monoidal categories $\Sh_X$ of spectrum-valued sheaves on $X$ and its dual in $\Mod{\Sp}$, the category $\mathbf{CoSh}_X$ of cosheaves on $X$.  Volpe uses this fact in \cite{mV2021} to construct a six functor formalism on the category of compactly generated Hausdorff spaces. For any continuous map $f: X \to Y$ the exceptional pushforward $f_!$ is determined using the commutative diagram 
\[\begin{tikzcd}
	{\Sh_X} & {\mathbf{CoSh}_X} \\
	{\Sh_Y} & {\mathbf{CoSh}_Y}
	\arrow["\simeq", from=1-1, to=1-2]
	\arrow["{f_!}"', from=1-1, to=2-1]
	\arrow["{f_*}", from=1-2, to=2-2]
	\arrow["\simeq"', from=2-1, to=2-2]
\end{tikzcd}\]
For any locally compact Hausdorff space $X$ the symmetric monoidal category $\Sh_X$  is \emph{locally rigid}, which entails that $\Sh_X$ canonically has the structure of Frobenius algebra in $\mathbf{Mod}_{\Sp}$ and is consequently canonically self-dual (see Example \ref{Hausdorff example}). In \cite{aKtN2024} it is observed that the self-duality $\Sh_X \simeq \mathbf{CoSh}_X$ is the one given by local rigidity, and thus that the construction of $f_!$ is completely canonical.  

It is well-known that the formula
$$f_* \simeq f_!$$
holds for $f$, a proper map in $C$, and 
$$f_! \simeq f_\sharp,$$
for $f$, an open embedding, where $f_\sharp$ denotes the left adjoint of $f^*$. In this note I define what it means for a symmetric monoidal functor $\mathcal{A} \leftarrow \mathcal{B}: f^*$ between locally rigid algebras to be proper or an open embedding, and derive the above formulas for such functors. It is expected that coefficient categories of six functor formalisms other than the one touched upon above are also locally rigid, opening up the possibility of defining exceptional pullback and pushforward functors and thus also establishing the above formulas in greater generality than previously possible. 

\subsubsection*{Organisation:}	This note consists of four sections: In \S \ref{Local rigidity} I recall some necessary background on locally rigid algebras. Then, in \S \ref{Proper functors} and \S \ref{Open embeddings} I prove the above formulas. Finally, in \S \ref{Parametrisation} I pose some questions about how to relate these results to parametrised homotopy theory. 

\subsubsection*{Notation:}	Denote by $\mathbf{CAlg}_{\mathbf{Pr}^\mathrm{L}}$ the symmetric monoidal $2$-category of presentably symmetric monoidal categories and colimit preserving symmetric monoidal functors. Reflecting the view towards six functor formalisms of the theory developed in this article, morphisms in $\mathbf{CAlg}_{\mathbf{Pr}^\mathrm{L}}$ are usually written as 
$$\mathcal{A} \leftarrow \mathcal{B}: f^*$$
with the right adjoint of $f^*$ denoted by $f_*: \mathcal{A} \to \mathcal{B}$ -- where $\mathcal{A}$ and $\mathcal{B}$ are to be thought of as coefficient categories associated to geometric objects $X,Y$ together with a map $f: X \to Y$. 

For a presentably symmetric monoidal category $\mathcal{A}$, the symmetric monoidal $2$-category of $\mathcal{A}$-modules and $\mathcal{A}$-linear functors is denoted by $\Mod{\mathcal{A}}$, and for any two $\mathcal{A}$-modules $\mathcal{M},\mathcal{N}$ the category of $\mathcal{A}$-linear functors a.k.a.\ morphisms of $\mathcal{A}$-modules $\mathcal{M} \to \mathcal{N}$ is denoted by $[\mathcal{M}, \mathcal{N}]_\mathcal{A}$. Moreover, I write $\mathbf{CAlg}_\mathcal{A}$ for $\big(\mathbf{CAlg}_{\mathbf{Pr}^\mathrm{L}}\big)_{\mathcal{A} / }$. 

\subsubsection*{Acknowledgements:}	I thank Maxime Ramzi for answering several questions, Rok Gregoric for a careful reading of a draft of this note, and Mitchell Riley for taking an interest in this project. 

\section{Local rigidity}	\label{Local rigidity} 

In this section I first recall some basic facts about locally rigid algebras and then define exceptional pushforward and pullback functors in \S \ref{exceptional}. 

Recall that a morphism $\mathcal{M} \leftarrow \mathcal{N}: f^*$ of $\mathcal{A}$-modules is an \Emph{$\mathcal{A}$-internal left adjoint} if its right adjoint $f_*: \mathcal{M} \to \mathcal{N}$ preserves colimits and satisfies the projection formula w.r.t.\ the $\mathcal{A}$-action, i.e., if for all objects $a$ in $\mathcal{A}$ and $m$ in $\mathcal{M}$ the natural map $a \otimes f_*(m) \to f_*(a \otimes m)$ is an isomorphism. 

\begin{definition}	\label{locally rigid} 
A morphism $\mathcal{A} \leftarrow \mathcal{B}$ in $\mathbf{CAlg}_{\mathbf{Pr}^\mathrm{L}}$ is \Emph{locally rigid} if
	\begin{enumerate}[label = (\alph*)]
	\item	\label{c1}	$\mathcal{A}$ is dualisable as an object of $\Mod{\mathbf{B}}$, and
	\item	\label{c2}	the multiplication map $\mathcal{A} \leftarrow \mathcal{A} \otimes_\mathcal{B} \mathcal{A}: \Delta^*$ is an $\mathcal{A} \otimes_\mathcal{B} \mathcal{A}$-internal left adjoint. 
	\end{enumerate} 
\qede
\end{definition} 

To my knowledge, the notion of local rigidity was first introduced in \cite{dG2015} (where it is called \emph{rigidity}). Locally rigid categories are discussed in detail in \cite{mR2024} and \cite[App.~C]{dAdGdGsRnRyV2020}. 

\begin{remark} 
It will be seen below that any locally rigid $\mathcal{B}$-algebra is canonically endowed with the structure of a Frobenius algebra in $\Mod{\mathcal{B}}$ (see \cite[Def.~4.6.5.1]{jL2012}). Condition \ref{c1} reflects the fact that any Frobenius algebra is canonically self dual. Condition \ref{c2} is equivalent to requiring that the functor $\Delta_*: \mathcal{A} \to \mathcal{A} \otimes_\mathcal{B} \mathcal{A}$ is $\mathcal{A}$-linear w.r.t.\ the two $\mathcal{A}$ actions on $\mathcal{A} \otimes_\mathcal{B} \mathcal{A}$, i.e., that the diagrams 
\[\begin{tikzcd}[column sep=large]
	{\mathcal{A} \otimes_\mathcal{B} \mathcal{A}} & {\mathcal{A} \otimes_\mathcal{B} \mathcal{A} \otimes_\mathcal{B} \mathcal{A}} & {\mathcal{A} \otimes_\mathcal{B}  \mathcal{A}} & {\mathcal{A} \otimes_\mathcal{B} \mathcal{A} \otimes_\mathcal{B} \mathcal{A}} \\
	{\mathcal{A}} & {\mathcal{A} \otimes_\mathcal{B} \mathcal{A}} & {\mathcal{A}} & {\mathcal{A} \otimes_\mathcal{B}  \mathcal{A}}
	\arrow["{\id \otimes\Delta_*}", from=1-1, to=1-2]
	\arrow["{\Delta^*}", from=1-1, to=2-1]
	\arrow["{\Delta^* \otimes\id}", from=1-2, to=2-2]
	\arrow["{\Delta_* \otimes\id}", from=1-3, to=1-4]
	\arrow["{\Delta^*}", from=1-3, to=2-3]
	\arrow["{\id \otimes\Delta^* }", from=1-4, to=2-4]
	\arrow["{\Delta_*}", from=2-1, to=2-2]
	\arrow["{\Delta_*}", from=2-3, to=2-4]
\end{tikzcd}\]
commute. By Proposition \ref{comultiplication} the functor $\Delta_*: \mathcal{A} \to \mathcal{A} \otimes_\mathcal{B} \mathcal{A}$ is the comultiplication associated to the Frobenius algebra structure on $\mathcal{A}$, so property \ref{c2} encodes the Frobenius law/condition of $\mathcal{A}$ (see \cite[Lm.~2.3.19]{jK2004}). \qede
\end{remark} 

Let $\mathcal{V} \leftarrow \mathcal{W}: p^*$ be a $\mathcal{W}$-internal left adjoint in $\mathbf{CAlg}_{\mathbf{Pr}^L}$ such that the essential image of $p^*$ generates $\mathcal{V}$ under colimits\footnote{It would probably be more natural to require $\mathcal{V}$ to be generated under $\mathcal{W}$-enriched colimits in some suitable sense.} then for any $\mathcal{W}$-module $\mathcal{P}$ the adjunction $\copadjunction{p_*: \mathcal{V}}{\mathcal{W}:p^*}$ induces adjunctions 
\begin{equation}	\label{free forget}
\copadjunction{\mathcal{P} = \mathcal{W} \otimes_\mathcal{W} \mathcal{P}}{\mathcal{V} \otimes_\mathcal{W} \mathcal{P}} \quad \text{and} \quad \copadjunction{\mathcal{P}  = [\mathcal{W}, \mathcal{P}]_\mathcal{W}}{[\mathcal{V}, \mathcal{P}]_\mathcal{W}},
\end{equation} 
both of which are monadic. The monads induced by these adjunctions are both equivalent to $p_*\mathbf{1}_\mathcal{V} \otimes \emptyinput$, so that $\mathcal{V} \otimes_\mathcal{W} \mathcal{P}$ and $[\mathcal{V}, \mathcal{P}]_\mathcal{W}$ are likewise equivalent (see \cite[Prop.~C.2.3]{dAdGdGsRnRyV2020} for details). 

\begin{remark} 
The functors $\mathcal{V} \otimes_\mathcal{W} \emptyinput$ and $[\mathcal{V}, \emptyinput]_\mathcal{W}$ are the left and right adjoint functors, respectively, of the restriction of scalars along $\mathcal{V} \leftarrow \mathcal{W}: p^*$. That these coincide in the above case could perhaps profitably be viewed as a form of ambidexterity, an important manifestation of duality in six functor formalisms. \qede
\end{remark} 

For the remainder of this section $\mathcal{A} \leftarrow \mathcal{B}: f^*$ denotes a locally rigid morphism in $\mathbf{CAlg}_{\mathbf{Pr}^L}$, and $(\emptyinput)^\vee$ denotes duality in $\Mod{\mathcal{B}}$ w.r.t.\ $\otimes_\mathcal{B}$. Applying the discussion above to  
	\begin{itemize}
	\item	$\mathcal{W}	 =	\mathcal{A} \otimes_\mathcal{B} \mathcal{A}$ 
	\item	$\mathcal{V}	=	\mathcal{A}$
	\item	$\mathcal{P} 	=	\mathcal{M}^\vee \otimes_\mathcal{B} \mathcal{N}$ for any $\mathcal{A}$-modules $\mathcal{M}$ and $\mathcal{N}$, with $\mathcal{M}$ dualisable over $\mathcal{B}$, and the action of $\mathcal{A} \otimes_\mathcal{B} \mathcal{A}$ on $\mathcal{M}^\vee \otimes_\mathcal{B} \mathcal{N}$ given by acting with each copy of $\mathcal{A}$ on $\mathcal{M}^\vee$ and $\mathcal{N}$, respectively,
	\end{itemize}
one obtains 
\begin{equation}	\label{locally rigid duality}
	\mathcal{M}^\vee \otimes_\mathcal{A} \mathcal{N}	\simeq	\mathcal{A} \otimes_{\mathcal{A} \otimes_\mathcal{B} \mathcal{A}} \mathcal{M}^\vee \otimes_\mathcal{B} \mathcal{N}
												\simeq	[\mathcal{A}, \mathcal{M}^\vee \otimes_\mathcal{B} \mathcal{N}]_{\mathcal{A} \otimes_\mathcal{B} \mathcal{A}}
												\simeq	[\mathcal{A}, [\mathcal{M}, \mathcal{N}]]_{\mathcal{A} \otimes_\mathcal{B} \mathcal{A}}
												\simeq	[\mathcal{M},\mathcal{N}]_\mathcal{A}
\end{equation} 
showing that $\mathcal{M}^\vee$ is not only dual to $\mathcal{M}$ as a $\mathcal{B}$-module, but also as an $\mathcal{A}$-module. Hence, setting $\mathcal{N} = \mathcal{A}$ yields 
\begin{equation}	\label{self-duality} 
	\mathcal{M}^\vee	\simeq	[\mathcal{M},\mathcal{A}]_\mathcal{A}.	
\end{equation} 

\begin{remark} 
The above discussion reflects a general phenomenon whereby for the locally rigid morphism $\mathcal{A} \leftarrow \mathcal{B}: f^*$, restriction of scalars both preserves and reflects properties of $\mathcal{A}$-modules. This phenomenon is explored in detail in \cite{mR2024}. \qede
\end{remark} 

The canonical equivalence $\mathcal{A}^\vee \simeq \mathcal{A}$ yields a commutative diagram 
\[\begin{tikzcd}
	{\mathcal{A}} & {\mathcal{A}} \\
	{\mathcal{A}^\vee \otimes_\mathcal{B} \mathcal{A}} & {\mathcal{A} \otimes_\mathcal{B} \mathcal{A}}
	\arrow["{ = }", from=1-1, to=1-2]
	\arrow["{\mathbf{1}_\mathcal{A} \mapsto\id_\mathcal{A}}"', from=1-1, to=2-1]
	\arrow["{\Delta_*}", from=1-2, to=2-2]
	\arrow[from=2-1, to=2-2]
\end{tikzcd}\]
and thus the counit of the self-duality of $\mathcal{A}$ is given by $\mathcal{B} \xrightarrow{f^*} \mathcal{A} \xrightarrow{\Delta_*} \mathcal{A} \otimes_\mathcal{B} \mathcal{A}$ (see \cite[Lm.~C.3.3]{dGdKnRyV2022}). 

\begin{proposition}[{\cite[Lm.~C.3.5]{dAdGdGsRnRyV2020}}]	\label{comultiplication} 
The $\mathcal{B}$-linear dual of $\Delta^*$  is given by $\Delta_*$. 
\end{proposition} 

\begin{proof} 
Under the canonical equivalences $[\mathcal{A} \otimes_\mathcal{B} \mathcal{A}, \mathcal{A}]_\mathcal{B} \simeq [\mathcal{A}, \mathcal{A} \otimes_\mathcal{B} \mathcal{A}]_\mathcal{B} \simeq [\mathcal{B}, \mathcal{A} \otimes_\mathcal{B} \mathcal{A} \otimes_\mathcal{B} \mathcal{A}]_\mathcal{B}$ given by duality in $\Mod{\mathcal{B}}$ w.r.t.\ $\otimes_\mathcal{B}$ and local rigidity, the functors $\Delta^*$ and $\Delta_*$ correspond to the functors $\mathcal{B} \to \mathcal{A} \otimes_\mathcal{B} \mathcal{A} \otimes_\mathcal{B} \mathcal{A}$ given by the two outer paths in the following diagram
\[\begin{tikzcd}[column sep=4em]
	{\mathcal{B}} & {\mathcal{A} \otimes_\mathcal{B} \mathcal{A}} & {\mathcal{A} \otimes_\mathcal{B} \mathcal{A} \otimes_\mathcal{B} \mathcal{A}} & {\mathcal{A} \otimes_\mathcal{B} \mathcal{A} \otimes_\mathcal{B} \mathcal{A} \otimes_\mathcal{B} \mathcal{A}} & {\mathcal{A} \otimes_\mathcal{B} \mathcal{A} \otimes_\mathcal{B} \mathcal{A} \otimes_\mathcal{B} \mathcal{A}} \\
	& {\mathcal{A}} & {\mathcal{A} \otimes_\mathcal{B} \mathcal{A}} & {\mathcal{A} \otimes_\mathcal{B} \mathcal{A} \otimes_\mathcal{B} \mathcal{A}}
	\arrow[from=1-1, to=1-2]
	\arrow[from=1-1, to=2-2]
	\arrow["{\Delta_* \otimes \id}", from=1-2, to=1-3]
	\arrow["{\Delta^*}", from=1-2, to=2-2]
	\arrow["{\id \otimes \id \otimes \Delta_*}", from=1-3, to=1-4]
	\arrow["{\id \otimes \Delta^*}", from=1-3, to=2-3]
	\arrow["{\sigma_{2,3}}", from=1-4, to=1-5]
	\arrow["{\id \otimes \Delta^* \otimes \id}", from=1-4, to=2-4]
	\arrow["{\Delta^* \otimes\  \id \otimes \id}", from=1-5, to=2-4]
	\arrow["{\Delta_*}", from=2-2, to=2-3]
	\arrow["{\id \otimes \Delta_*}", from=2-3, to=2-4]
\end{tikzcd}\]
where $\sigma_{i,j}$ exchanges the $i$-th and $j$-th factors.

The leftmost triangle as well as the two squares commute. Unfortunately, the rightmost triangle does not. Using the parts of the preceding diagram which do commute as well as the cocommutativity and coassociativity of $\Delta_*$ it easily checked that the following diagram 
\[\begin{tikzcd}[column sep=4em]
	{\mathcal{B}} & {\mathcal{A} \otimes_\mathcal{B} \mathcal{A}} & {\mathcal{A} \otimes_\mathcal{B} \mathcal{A} \otimes_\mathcal{B} \mathcal{A}} & {\mathcal{A} \otimes_\mathcal{B} \mathcal{A} \otimes_\mathcal{B} \mathcal{A} \otimes_\mathcal{B} \mathcal{A}} \\
	& {\mathcal{A}} & {\mathcal{A} \otimes_\mathcal{B} \mathcal{A}} & {\mathcal{A} \otimes_\mathcal{B} \mathcal{A} \otimes_\mathcal{B} \mathcal{A}} & {\mathcal{A} \otimes_\mathcal{B} \mathcal{A} \otimes_\mathcal{B} \mathcal{A}}
	\arrow[from=1-1, to=1-2]
	\arrow[from=1-1, to=2-2]
	\arrow["{\Delta_* \otimes \id}", from=1-2, to=1-3]
	\arrow["{\Delta^*}", from=1-2, to=2-2]
	\arrow["{\sigma_{1,2} \otimes \Delta_*}", from=1-3, to=1-4]
	\arrow["{\id \otimes \Delta^* \otimes \id}", from=1-4, to=2-4]
	\arrow["{\Delta_*}", from=2-2, to=2-3]
	\arrow["{\Delta_* \otimes \id }", from=2-3, to=2-4]
	\arrow["{\sigma_{1,2}}", from=2-4, to=2-5]
\end{tikzcd}\]
commutes and has outer paths canonically equivalent to those of the first diagram.
\end{proof} 

Thus, the dual of the counit -- which is equivalent to the unit -- is given by $\mathcal{A} \otimes_\mathcal{B} \mathcal{A} \xrightarrow{\Delta^*} \mathcal{A} \xrightarrow{(f^*)^\vee} \mathcal{B}$, exhibiting $\mathcal{A}$ as a \emph{Frobenius algebra} over $\mathcal{B}$ (see \cite[Def.~4.6.5.1]{jL2012}). 

Observe that $(f^*)^\vee$ corresponds to the image of $\mathbf{1}_\mathcal{B}$ under $\mathcal{B} = [\mathcal{B}, \mathcal{B}]_\mathcal{B} \xrightarrow{\simeq} [\mathcal{B}, \mathcal{A}]_\mathcal{A}$. This morphism is called the \Emph{exceptional global sections functor} and is denoted by $\Gamma_{\mathcal{B}!}^\mathcal{A}$ (or $\Gamma_{\mathcal{B}!}$, when $\mathcal{A}$ is clear from context). 

\begin{example}[{\cite[\S 6]{kA2023v}}] \label{Hausdorff example} Let $X$ be a locally compact Hausdorff space, then the category of spectrum-valued sheaves $\Sh_X$ is a locally rigid $\Sp$-algebra. By \cite[Th.~5.5.5.1]{jL2012} the category $\Sh_X$ is dualisable. The equivalence $\Sh_X \otimes_{\Sp} \Sh_X = \Sh_{X \times X}$ (see \cite[Ex.~4.8.1.19]{jL2012}) together with the Hausdorff condition on $X$ ensures that the functor $\Sh_X \xleftarrow{\Delta^*} \Sh_X  \otimes_\Sp \Sh_X $ is an $\Sh_X  \otimes_\Sp \Sh_X $-internal right adjoint. The exceptional global sections functor is given by global sections with compact support, explaining the notation above. \qede
\end{example} 

\begin{remark}	\label{LCH}
The notion of a locally rigid algebra $\mathcal{A} \leftarrow \mathcal{B}$ is itself reminiscent of a locally compact Hausdorff space: Condition \ref{c1} says that $\mathcal{A}$ is locally compactly generated. Indeed, an $\Sp$-module is dualisable precisely when it is generated under colimits by compactly exhaustible objects (see \cite[Th.~2.36]{mR2024}), and for $\Sh_X$, an open subset of $X$ is compactly exhaustible if it can be written as the union of an increasing sequence of compact subsets. Condition \ref{c2} encodes separatedness, as the multiplication map is proper in the sense of \S \ref{Proper functors} (see also Question \ref{question}). \qede
\end{remark} 

I will briefly discuss two special classes of locally rigid algebras, further illustrating the analogy between such algebras and locally compact Hausdorff spaces. 

\begin{proposition}	\label{rigid equivalent} 
The following are equivalent: 
	\begin{enumerate}[label = {\normalfont(\Roman*)}]
	\item The functor $\mathcal{A} \leftarrow \mathcal{B}:f^*$ admits a $\mathcal{B}$-internal right adjoint. 
	\item	There is an equivalence $\Gamma_{\mathcal{A}!} \simeq f_*$.
	\end{enumerate} 
\end{proposition} 

\begin{proof} 
Assuming (I), it is shown explicitly that $\mathcal{B} \xrightarrow{f^*} \mathcal{A} \xrightarrow{\Delta_*} \mathcal{A} \otimes_\mathcal{B} \mathcal{A}$ and $\mathcal{A} \otimes_\mathcal{B} \mathcal{A} \xrightarrow{\Delta^*} \mathcal{A} \xrightarrow{f_*} \mathcal{B}$ provide dualisability data in  \cite[\S 2.2]{mHpSsSnS2021}. Thus, $\mathcal{A} \otimes_\mathcal{B} \mathcal{A} \xrightarrow{\Delta^*} \mathcal{A} \xrightarrow{f_*} \mathcal{B}$ is equivalent to the counit $\mathcal{A} \otimes_\mathcal{B} \mathcal{A} \xrightarrow{\Delta^*} \mathcal{A} \xrightarrow{\Gamma_{\mathcal{A}!}} \mathcal{B}$. As $\mathcal{A} \otimes_\mathcal{B} \mathcal{A} \xrightarrow{\Delta^*} \mathcal{A}$ admits a section, it follows that $\Gamma_{\mathcal{A}!} \simeq f_*$. 

The converse follows from Proposition \ref{linear} below (which is independent of this proposition). 
\end{proof} 

\begin{definition} 
The morphism $\mathcal{A} \leftarrow \mathcal{B}:f^*$ is \Emph{rigid} if it satisfies the equivalent definitions of Proposition \ref{rigid equivalent}. \qede
\end{definition} 

\begin{remark} 
The analogy between locally compact Hausdorff spaces and locally rigid algebras in Remark \ref{LCH} specialises to one between rigid algebras and  \emph{compact} Hausdorff spaces, as the morphism $\mathcal{A} \leftarrow \mathcal{B}:f^*$ is proper in the sense described in \S \ref{Proper functors}. 

Another interesting special case is when $\mathcal{A}$ is compactly generated. By \cite{oH2023} this is analogous to a totally disconnected locally compact Hausdorff space. \qede
\end{remark} 

\begin{remark} 
In \cite[App.~F]{aE2024} Efimov explains a different, more detailed analogy between dualisable $\Sp$-modules and compact Hausdorff spaces. \qede
\end{remark} 

Before moving on to exceptional pushforwards and pullbacks in \S \ref{exceptional} I will add some clarifications about the relationship between duality over $\mathcal{A}$ and duality over $\mathcal{B}$. 

Recall that restriction of scalars $\Mod{\mathcal{A}} \to \Mod{\mathcal{B}}$ has a right adjoint given by $[\mathcal{A}, \mathcal{N}]_\mathcal{B} \mapsfrom \mathcal{N}$, so that one obtains the following natural equivalences for any $\mathcal{A}$-module $\mathcal{M}$
\begin{equation}	\label{Gamma}
\begin{array}{lllll}
[\mathcal{M},\mathcal{A}]_\mathcal{A}	&	\xrightarrow{\simeq}	&	[\mathcal{M},[\mathcal{A},\mathcal{B}]_\mathcal{B}]_\mathcal{A}			&	\xrightarrow{\simeq}	&	[\mathcal{M},\mathcal{B}]_\mathcal{B}	\\
\lambda							&	\mapsto			&	\Gamma_{\mathcal{A}!}\big(\lambda(\emptyinput)\otimes(\emptyinput)\big)	&	\mapsto			&	\Gamma_{\mathcal{A}!} \circ \lambda.
\end{array}
\end{equation} 

\begin{proposition} 
For any $\mathcal{A}$-modules $\mathcal{M}, \mathcal{N}$ with $\mathcal{M}$ dualisable as a $\mathcal{B}$-module, the equivalence given by composing 
$$	 [\mathcal{M}, \mathcal{N}]_\mathcal{A}	\xrightarrow{\simeq}											[\mathcal{M}, \mathcal{A}]_\mathcal{A} \otimes_\mathcal{A} \mathcal{N}	
									\xrightarrow{(\Gamma_{\mathcal{A} !} \circ \emptyinput , \id_\mathcal{N})}	[\mathcal{M}, \mathcal{B}]_\mathcal{B} \otimes_\mathcal{A} \mathcal{N}	$$
is equivalent to (\ref{locally rigid duality}). 
\end{proposition} 

\begin{proof} 
Clearly, the two equivalences are the same when $\mathcal{M} =\mathcal{N} =  \mathcal{A}$. The case $\mathcal{N} = \mathcal{A}$ may then be deduced from the naturality of (\ref{locally rigid duality}), as for any object $m$ in $\mathcal{M}$ one obtains a commutative diagram 
\[\begin{tikzcd}
	{[\mathcal{A}, \mathcal{A}]_\mathcal{A}} & {[\mathcal{A}, \mathcal{A}]_\mathcal{A} \otimes_\mathcal{A} \mathcal{A}} & {[\mathcal{A}, \mathcal{B}]_\mathcal{B} \otimes_\mathcal{A} \mathcal{A}} \\
	{[\mathcal{M}, \mathcal{A}]_\mathcal{A}} & {[\mathcal{M}, \mathcal{A}]_\mathcal{A} \otimes_\mathcal{A} \mathcal{A}} & {[\mathcal{M}, \mathcal{B}]_\mathcal{B} \otimes_\mathcal{A} \mathcal{A}}
	\arrow["\simeq", from=1-1, to=1-2]
	\arrow["\simeq", from=1-2, to=1-3]
	\arrow[from=2-1, to=1-1]
	\arrow["\simeq", from=2-1, to=2-2]
	\arrow[from=2-2, to=1-2]
	\arrow["\simeq", from=2-2, to=2-3]
	\arrow[from=2-3, to=1-3]
\end{tikzcd}\]
where the vertical arrows are given by precomposing with $m: \mathcal{A} \to \mathcal{M}$. Finally, for the general case it is enough to observe that any $\mathcal{A}$-module $\mathcal{N}$ may be written as a colimit of diagram taking limits in the subcategory of $\Mod{\mathcal{A}}$ spanned by $\mathcal{A}$, and that $\otimes_\mathcal{A}$ commutes with colimits.  
\end{proof} 

By the construction of (\ref{locally rigid duality}) one obtains the following corollary:

\begin{corollary}	\label{triangle}
For any $\mathcal{A}$-modules $\mathcal{M}, \mathcal{N}$ with $\mathcal{M}$ dualisable as a $\mathcal{B}$-module the diagram 
\[\begin{tikzcd}
	& {[\mathcal{M}, \mathcal{B}]_\mathcal{B} \otimes_\mathcal{B}\mathcal{N}} \\
	{[\mathcal{M}, \mathcal{A}]_\mathcal{A} \otimes_\mathcal{A} \mathcal{N}} && {[\mathcal{M}, \mathcal{B}]_\mathcal{B} \otimes_\mathcal{A} \mathcal{N}}
	\arrow[from=1-2, to=2-1]
	\arrow[from=1-2, to=2-3]
	\arrow["{(\Gamma_{\mathcal{V} !} \circ \emptyinput , \id_\mathcal{N})}", from=2-1, to=2-3]
\end{tikzcd}\]
commutes, where the diagonal arrows are the monadic left adjoints of (\ref{free forget}) applied to the appropriate choice of $\mathcal{V}$ and $\mathcal{W}$ as in the discussion preceding (\ref{locally rigid duality}). 
\end{corollary} 

\subsection{Exceptional pullbacks and pushforward functors}	\label{exceptional}

Let $\mathcal{A} \xleftarrow{f^*} \mathcal{B} \xleftarrow{g^*} \mathcal{C}$ be morphisms in $\mathbf{CAlg}_{\mathbf{Pr}^L}$ such that $\mathcal{A}$ and $\mathcal{B}$ are locally rigid over $\mathcal{C}$. In this subsection $(\emptyinput)^\vee$ denotes the duality functor in $\Mod{\mathcal{C}}$. 

Consider a $\mathcal{C}$-linear functor between dualisable $\mathcal{C}$-modules $\mathcal{M} \leftarrow \mathcal{N}: h^*$, then its image under the canonical equivalence $[\mathcal{N}, \mathcal{M}]_\mathcal{C} \simeq [\mathcal{M}^\vee, \mathcal{N}^\vee]_\mathcal{C}$ provided by duality in $\Mod{\mathcal{C}}$ w.r.t.\ $\otimes_\mathcal{C}$ may be written explicitly as the functor $(h^*)^\vee: \mathcal{M}^\vee \to \mathcal{N}^\vee, \enskip (\lambda: \mathcal{M} \to \mathcal{C}) \mapsto (\lambda \circ h^*: \mathcal{N} \to \mathcal{C}$). The \Emph{exceptional pushforward} $f_!: \mathcal{A} \to \mathcal{B}$ is the unique functor making the diagram
\[\begin{tikzcd}
	{\mathcal{A}} & {[\mathcal{A},\mathcal{C}]_\mathcal{C}} \\
	{\mathcal{B}} & {[\mathcal{B},\mathcal{C}]_\mathcal{C}}
	\arrow["\simeq", from=1-1, to=1-2]
	\arrow["{f_!}"', from=1-1, to=2-1]
	\arrow["{(f^*)^\vee}", from=1-2, to=2-2]
	\arrow["\simeq", from=2-1, to=2-2]
\end{tikzcd}\]

commute, where the horizontal maps are given by (\ref{self-duality}). Observe that the equivalence $\mathcal{A} \xrightarrow{\simeq} [\mathcal{A},\mathcal{C}]_\mathcal{C}$ is given by $a \mapsto \Gamma_{\mathcal{A}!}(a \otimes_\mathcal{A} \emptyinput)$, so that $f_!$ is uniquely characterised as the functor making the formula 
\begin{equation}	\label{exceptional characterisation}
\Gamma_{\mathcal{B}!}(f_!(a) \otimes_\mathcal{B} \emptyinput) \simeq \Gamma_{\mathcal{A}!}(a \otimes_\mathcal{A} f^*(\emptyinput))
\end{equation} 
hold. The functor $(f^*)^\vee$ is cocontinuous, and thus also $f_!$. The right adjoint of $f_!$ is called the \Emph{exceptional pullback} of $f^*$ and is denoted by $f^!$. 

\begin{remark} 
Let $f: X \to Y$ be a morphism in a suitable site of geometric objects, equipped with a six functor formalism valued in locally rigid algebras, then using the above construction it is not necessary to be able to factor $f$ into an open embedding followed by a proper map in order to define the exceptional pushforward functor $f_!$ (which is how it is usually constructed). \qede
\end{remark} 

\begin{proposition}	\label{linear} 
The exceptional pushforward $f_!$ is $f^*$-linear. 
\end{proposition} 

\begin{proof} 
Using (\ref{exceptional characterisation}), for all objects $a$ in $\mathcal{A}$ and $b$ in $\mathcal{B}$ one has 
	$$	\Gamma_{\mathcal{B}!}(b \otimes_\mathcal{B} f_!(a) \otimes_\mathcal{B} \emptyinput)	\simeq	\Gamma_{\mathcal{A}!}(a \otimes_\mathcal{B} f^*(b \otimes_\mathcal{B} \emptyinput))
																			\simeq	\Gamma_{\mathcal{A}!}(f^*(b) \otimes_\mathcal{A} a \otimes_\mathcal{A} f^*(\emptyinput)).$$
\end{proof}

\section{Proper functors}	\label{Proper functors}

Throughout this section $\mathcal{A} \xleftarrow{f^*} \mathcal{B} \xleftarrow{g^*} \mathcal{C}$ denote morphisms in $\mathbf{CAlg}_{\mathbf{Pr}^L}$, with $\mathcal{A}$ and $\mathcal{B}$ locally rigid over $\mathcal{C}$. Moreover, $f_!$ denotes the $\mathcal{C}$-linear dual of $f^*$ (see \S \ref{exceptional}). 

\begin{definition} 
The functor $\mathcal{A} \xleftarrow{f^*} \mathcal{B}$ is \Emph{proper} if its right adjoint is $\mathcal{B}$-linear. \qede
\end{definition} 

The goal of this section is to prove the following theorem, justifying the term \emph{proper} by deducing the characteristic property of proper maps in the context of six functor formalisms: 

\begin{theorem}	\label{proper pushforward} 
If $f^*$ is proper, then there is a canonical equivalence $f_! \simeq f_*$. 
\end{theorem}

The proof of this theorem requires two preliminary results. 

\begin{proposition} 
Assume that $\mathcal{A} \leftarrow \mathcal{B}: f^*$ is locally rigid, then there is a canonical equivalence  $\Gamma_{\mathcal{A}!}^\mathcal{C} \simeq \Gamma_{\mathcal{B}!}^\mathcal{C} \circ \Gamma_{\mathcal{A}!}^\mathcal{B}$. 
\end{proposition} 

\begin{proof} 
By (\ref{Gamma}) the proposition is equivalent to the statement that the uppermost $2$-cell in the diagram 
\[\begin{tikzcd}
	{[\mathcal{A},\mathcal{A}]_\mathcal{A} \otimes_\mathcal{A} \mathcal{A}} & {[\mathcal{A},\mathcal{B}]_\mathcal{B} \otimes_\mathcal{A}\mathcal{A}} & {[\mathcal{A},\mathcal{C}]_\mathcal{C} \otimes_\mathcal{A}\mathcal{A}} \\
	& {[\mathcal{A},\mathcal{B}]_\mathcal{B} \otimes_\mathcal{B}\mathcal{A}} & {[\mathcal{A},\mathcal{C}]_\mathcal{C} \otimes_\mathcal{B}\mathcal{A}} \\
	& {[\mathcal{A},\mathcal{C}]_\mathcal{C} \otimes_\mathcal{C}\mathcal{A}}
	\arrow[from=1-1, to=1-2]
	\arrow[curve={height=-18pt}, from=1-1, to=1-3]
	\arrow[from=1-2, to=1-3]
	\arrow[from=2-2, to=1-1]
	\arrow[from=2-2, to=1-2]
	\arrow[from=2-2, to=2-3]
	\arrow[from=2-3, to=1-3]
	\arrow[from=3-2, to=1-1]
	\arrow[controls={+(6,0) and +(2,-1)}, from=3-2, to=1-3]
	\arrow[from=3-2, to=2-2]
	\arrow[from=3-2, to=2-3]
\end{tikzcd}\]
commutes.

The arrows in the left- and rightmost cells are the monadic left adjoints of (\ref{free forget}) applied to appropriate choices of $\mathcal{V}$ and $\mathcal{W}$ as in the discussion preceding (\ref{locally rigid duality}). 

I will show that the precompositions of both $\Gamma_{\mathcal{A}!}^\mathcal{C}$ and $\Gamma_{\mathcal{B}!}^\mathcal{C} \circ \Gamma_{\mathcal{A}!}^\mathcal{B}$ with the leftmost arrow are equivalent to the rightmost arrow, so that the proposition follows from the Barr-Beck-Lurie theorem together with the universal property of the category of algebras over a monad (see \cite[\S 10.2]{eRdVe}). 

The outer triangle commutes by Corollary \ref{triangle}. It thus remains to show that the triangle formed by $\Gamma_{\mathcal{B}!}^\mathcal{C} \circ \Gamma_{\mathcal{A}!}^\mathcal{B}$ together with the leftmost and rightmost arrows commutes, which in turn follows from showing that the cells contained in this triangle commute: The square commutes by the functoriality of tensoring. The triangles bordering the square to the left and below commute by Corollary \ref{triangle}. The remaining two triangles commute as they exhibit, respectively, the composition of left and right adjoints to restrictions of scalars. 
\end{proof} 

\begin{proposition} 
Assume that $\mathcal{A} \leftarrow \mathcal{B}:f^*$ admits a $\mathcal{C}$-internal right adjoint, then $\mathcal{A} \leftarrow \mathcal{B}$ is a locally rigid $\mathcal{B}$-algebra. 
\end{proposition} 

\begin{proof} 
As $\mathcal{A}$ is dualisable as a $\mathcal{C}$-module it is also dualisable as a $\mathcal{B}$-module by (\ref{locally rigid duality}). It thus remains to verify property \ref{d2} in Definition \ref{locally rigid}. By \cite[Props.~C.5.5~\&~C.6.4]{dAdGdGsRnRyV2020} it suffices to prove that $\mathcal{A} \otimes_\mathcal{B} \mathcal{A} \to \mathcal{A}$ is a $\mathcal{C}$-linear left adjoint which I deduce from the following claim. \\

\noindent \underline{Claim:} The quotient map $\mathcal{A} \otimes_\mathcal{C} \mathcal{A} \to \mathcal{A} \otimes_\mathcal{B} \mathcal{A}$ admits a $\mathcal{B} \otimes_\mathcal{C} \mathcal{B}$-internal and conservative right adjoint.	\\

Thus, $\mathcal{A} \otimes_\mathcal{B} \mathcal{A} \to \mathcal{A}$ is a $\mathcal{B} \otimes_\mathcal{C} \mathcal{B}$-internal (so a fortiori $\mathcal{C}$-internal) left adjoint by \cite[Lm.~1.30]{mR2024} and the fact that the composition of 
	$$\mathcal{A} \otimes_\mathcal{C} \mathcal{A} \to \mathcal{A} \otimes_\mathcal{B} \mathcal{A} \to \mathcal{A}$$
is an $\mathcal{A} \otimes_\mathcal{C} \mathcal{A}$-internal (so a fortiori $\mathcal{B} \otimes_\mathcal{C} \mathcal{B}$-internal) left adjoint. \\

\noindent \underline{Proof of claim:} The quotient map $\mathcal{A} \otimes_\mathcal{C} \mathcal{A} \to \mathcal{A} \otimes_\mathcal{B} \mathcal{A}$ may be written as 
\begin{equation}	\label{quotient} 
(\mathcal{A} \otimes_\mathcal{C} \mathcal{A}) \otimes_{\mathcal{B} \otimes_\mathcal{C} \mathcal{B}} (\mathcal{B} \otimes_\mathcal{C} \mathcal{B}) \to (\mathcal{A} \otimes_\mathcal{C} \mathcal{A}) \otimes_{\mathcal{B} \otimes_\mathcal{C} \mathcal{B}} \mathcal{B}.
\end{equation} 
As  $\mathcal{B}$ is locally rigid, (\ref{quotient}) admits a $\mathcal{B} \otimes_\mathcal{C} \mathcal{B}$-internal left adjoint by \cite[Cor.~1.32]{mR2024}, and because the image of (\ref{quotient}) generates $(\mathcal{A} \otimes_\mathcal{C} \mathcal{A}) \otimes_{\mathcal{B} \otimes_\mathcal{C} \mathcal{B}} \mathcal{B}$ under colimits, this right adjoint is conservative. \end{proof} 

\begin{proof}[Proof of Theorem \ref{proper pushforward}]
By Proposition \ref{rigid equivalent} one has $\Gamma_{\mathcal{A}!}^\mathcal{B} \simeq f_*$, so the theorem follows by combining the preceding two results. 
\end{proof} 

\begin{question} 
Under duality in $\Mod{\mathcal{C}}$, the functors $\Gamma_{\mathcal{A}!}^\mathcal{B}$ and $f^*$  correspond to functors $\mathcal{C} \to \mathcal{B} \otimes_\mathcal{C} \mathcal{A}$ given by the upper and lower paths of the following diagram:  
\[\begin{tikzcd}
	{\mathcal{C}} & {\mathcal{B}} && {\mathcal{A}} \\
	&&& {\mathcal{A} \otimes_\mathcal{C} \mathcal{A}} \\
	& {\mathcal{B} \otimes_\mathcal{C} \mathcal{B}} && {\mathcal{B} \otimes_\mathcal{C} \mathcal{A}}
	\arrow["{g^*}", from=1-1, to=1-2]
	\arrow["{f^*}", from=1-2, to=1-4]
	\arrow["{(\mu_\mathcal{B})_*}"', from=1-2, to=3-2]
	\arrow["{(\mu_\mathcal{A})_*}", from=1-4, to=2-4]
	\arrow["{\Gamma_{\mathcal{A}!}^\mathcal{B} \otimes \id}", from=2-4, to=3-4]
	\arrow[shorten <=12pt, shorten >=12pt, Rightarrow, from=3-2, to=1-4]
	\arrow["{\id \otimes f^*}", from=3-2, to=3-4]
\end{tikzcd}\]
Assume that $f^*$ admits a $\mathcal{B}$-internal right adjoint, so that $\Gamma_{\mathcal{A}!}^\mathcal{B} \simeq f_*$ by Proposition \ref{rigid equivalent}. Then by the symmetric monoidality of $f^*$ one obtains a mate transformation in the above diagram. Transposing to functors $\mathcal{A} \to \mathcal{B}$ yields a natural transformation $f_! \Rightarrow f_*$. Is this natural transformation the equivalence of Theorem \ref{proper pushforward}?  \qede
\end{question}

\section{Open embeddings}	\label{Open embeddings}

Let $j: U \hookrightarrow X$ be an open embedding of locally compact Hausdorff spaces. Then the exceptional pushforward functor $j_!: \Sh_U \to \Sh_X$ may be identified with $j_\sharp$, the left adjoint of $j^*$. In this section I will define open embeddings between locally rigid algebras and derive the above identification for such functors. All statements will be about algebras over $\Sp$, as the proof of Theorem \ref{oe} uses the theory of recollements, which to my knowledge has not been developed over other commutative algebras in $\mathbf{Pr}^L$. 

\begin{definition}[{\cite[Def.~2.20]{jS2022}}]	\label{recollement} 
A \Emph{recollement} is a diagram of adjoints between $\Sp$-algebras
\[\begin{tikzcd}
	{\mathcal{U}} & {\mathcal{X}} & {\mathcal{Z}}
	\arrow["{j_\sharp}", shift left=4.5, hook, from=1-1, to=1-2]
	\arrow["{j_*}"', shift right=1.5, hook, from=1-1, to=1-2]
	\arrow["{j^*}"{description}, shift right=1.5, from=1-2, to=1-1]
	\arrow["{i^*}", shift left=1.5, from=1-2, to=1-3]
	\arrow["{i^!}"', shift right=4.5, from=1-2, to=1-3]
	\arrow["{i_*}"{description}, shift left=1.5, hook', from=1-3, to=1-2]
\end{tikzcd}\]
such that 
	\begin{enumerate}[label = (\alph*)]
	\item	$i^*,j^*$ are symmetric monoidal, 
	\item	$\mathbf{Im}_{j_\sharp} = \mathbf{Ker}_{i^*}	\quad \mathbf{Ker}_{j^*} = \mathbf{Im}_{i_*}	\quad	\mathbf{Im}_{j_*} = \mathbf{Ker}_{i^!}$
	\end{enumerate} 
\qede
\end{definition} 

\begin{remark}	\label{identified} 
Observe that any of the six functors in a recollement determines all the others. \qede
\end{remark} 

\begin{definition}
A morphism $\mathcal{U} \leftarrow \mathcal{X}: j^*$ of $\Sp$-algebras is called an \Emph{open embedding} if it extends to a recollement with $j^*$ as in Definition \ref{recollement}. A functor $i^*: \mathcal{X} \to \mathcal{Z}$ is called a \Emph{closed embedding} if it extends to a recollement with $i^*$ as in Definition \ref{recollement}. \qede
\end{definition} 

\begin{proposition}[{\cite[Prop.~2.34]{jS2022}}]	\label{linear recollement}
The functors $i_*$ and $j_\sharp$ are $\mathcal{X}$-linear. 	\qed
\end{proposition} 

\begin{theorem}	\label{oe} 
Let $\mathcal{U} \leftarrow \mathcal{X}: j^*$ be an open embedding, then there exists a canonical equivalence $j_! \simeq j_\sharp$. 
\end{theorem} 

\begin{proof} 
Consider the diagram
\[\begin{tikzcd}[column sep=large]
	{[\mathcal{U},\Sp]^L} & {[\mathcal{X},\Sp]^L} & {[\mathcal{Z},\Sp]^L}
	\arrow["{(j^*)^\vee}", shift left=4.5, hook, from=1-1, to=1-2]
	\arrow["{((j_\sharp)^*)^R}"', shift right=1.5, hook, from=1-1, to=1-2]
	\arrow["{(j_\sharp)^\vee}"{description}, shift right=1.5, from=1-2, to=1-1]
	\arrow["{(i_*)^\vee}", shift left=1.5, from=1-2, to=1-3]
	\arrow["{((i^*)^\vee)^R}"', shift right=4.5, from=1-2, to=1-3]
	\arrow["{(i^*)^\vee}"{description}, shift left=1.5, hook', from=1-3, to=1-2]
\end{tikzcd}\]
Under the equivalences $\mathcal{U} \simeq [\mathcal{U},\Sp]^L, \;  \mathcal{X} \simeq [\mathcal{X},\Sp]^L, \;  \mathcal{Z} \simeq [\mathcal{Z},\Sp]^L$, the functor $i_*$ may be identified with $(i^*)^\vee$, because $i_! \simeq i_*$ by Proposition \ref{linear recollement} and Theorem \ref{proper pushforward}, so that $j_\sharp$ becomes identified with $(j^*)^\vee$ by Remark \ref{identified}, yielding the desired identification $j_! \simeq  j_\sharp$. 
\end{proof} 

\begin{question}	\label{question} 
In various geometric contexts a morphism $X \to Y$ is said to be \emph{separated} if the diagonal map $X \to X \times_Y X$ is a closed embedding. In \S \ref{Local rigidity} I argued that condition \ref{d2} in Definition \ref{locally rigid} could be viewed as form of separatedness. An alternative notion of separatedness for algebras in $\mathbf{Pr}^L$ more closely analogous to familiar geometric situations would thus be to require that the multiplication map in Definition \ref{locally rigid} be a closed embedding rather than just proper. Are there interesting examples of such algebras? What good properties do such algebras possess?  Or is the multiplication map of a locally rigid algebra automatically a closed embedding?  \qede
\end{question} 

\section{Parametrisation}	\label{Parametrisation}

Let $C$ be a site, and $\mathbf{D}^*: C^{op} \to \mathbf{CAlg}_{\mathcal{A}}$ a sheaf taking values in the subcategory of $\mathbf{CAlg}_{\mathcal{A}}$ spanned by locally rigid $\mathcal{A}$-algebras (where $\mathcal{A}$ is usually $\Sp$). Then $\mathbf{D}_! \defeq [\mathbf{D}^*, \mathcal{A}]^L: C \to \mathbf{Pr}_{\mathcal{A}}^L$ is a cosheaf, such that for each object $X$ in $C$ there is a canonical equivalence $\mathbf{D}^*X \simeq \mathbf{D}_!X$. I will write $\mathbf{D}X$ to mean $\mathbf{D}^*X$ or equivalently $\mathbf{D}_!X$. The promotion of the pair $\mathbf{D}^*, \mathbf{D}_!$ to a six functor formalism necessitates the following two prerequisites (with which all other parts of being a six functor formalism are properties, as explained in \cite{aK2023}): 
\begin{enumerate}[label = (\alph*)]
\item	\label{d1}	For each morphism $f: X \to Y$ in $C$ the functor $f_!: \mathbf{D}X \to \mathbf{D}Y$ satisfies the projection formula with respect to $f^*$. 
\item \label{d2}	For each pullback square 
\[\begin{tikzcd}
	{X'} & X \\
	{Y'} & Y
	\arrow["g", from=1-1, to=1-2]
	\arrow["p", from=1-1, to=2-1]
	\arrow["q", from=1-2, to=2-2]
	\arrow["f", from=2-1, to=2-2]
\end{tikzcd}\]
in $C$ the square 
\[\begin{tikzcd}
	{\mathbf{D}X'} & {\mathbf{D}X} \\
	{\mathbf{D}Y'} & {\mathbf{D}Y}
	\arrow["{g_!}", from=1-1, to=1-2]
	\arrow["{p^*}"', from=2-1, to=1-1]
	\arrow["{f_!}", from=2-1, to=2-2]
	\arrow["{q^*}"', from=2-2, to=1-2]
\end{tikzcd}\]
commutes. 
\end{enumerate} 
By Proposition \ref{linear}, prerequisite \ref{d1} is a condition, which is automatically satisfied. Prerequisite \ref{d2} constitutes extra structure which, together with \ref{d1}, should encode that $\mathbf{D}_!$ is adequately \emph{$C$-parametrisedly} $\mathbf{D}^*$-linearly dual to $\mathbf{D}^*$. It would be beneficial to make sense of what this means precisely. One difficulty which would have to be overcome is that locally rigid $\mathcal{A}$-algebras are not closed under limits in $\mathbf{CAlg}_{\mathcal{A}}$, so that the extension of $\mathbf{D}^*: C^{op} \to \mathbf{CAlg}_{\mathcal{A}}$ to $\Sh_C$ no longer necessarily takes values in locally rigid $\mathcal{A}$-algebras for non representable sheaves on $C$. 

\bibliographystyle{/usr/local/texlive/2023/texmf-dist/bibtex/bst/base/alpha2} 
\bibliography{/Users/adrianclough/Mathematics/My_Mathematics/Ancillary_files/Documents/Adrians_bibliography}

\end{document}